\documentclass[12pt,english]{article}
\usepackage[margin=1.5in]{geometry}

\usepackage{mdwlist}
\usepackage{enumerate}
\usepackage{amssymb,amsbsy,latexsym}
\usepackage{amsmath}
\usepackage{graphics, subfigure, float}

\usepackage[T1]{fontenc} 
\usepackage{fourier}
\usepackage{bm}

\usepackage{amscd,amsthm}

\usepackage{pst-all}
\usepackage{pstricks-add}
\usepackage{pst-func}
\newpsobject{showgrid}{psgrid}{subgriddiv=1,griddots=10,gridlabels=6pt}
\usepackage{verbatim, comment}
\usepackage{datetime}

\theoremstyle{theorem}
\newtheorem{theorem}{Theorem}
\newtheorem{conjecture}{Conjecture}
\newtheorem{corollary}[theorem]{Corollary}
\newtheorem{lemma}[theorem]{Lemma}
\newtheorem{definition}[theorem]{Definition}
\newtheorem*{lemma*}{Lemma}

\providecommand{\setN}{\mathbb{N}}

\providecommand{\setR}{\mathbb{R}}

\newcommand{\E}{\mathop{\mathbb{E}}}

\newcommand{\Vol}{\textrm{Vol}}
\newcommand{\disc}{\textrm{disc}}

\newcommand{\vb}{\textrm{vb}}
\newcommand{\tr}{\textrm{Tr}}
\newcommand{\spn}{\textrm{span}}

\usepackage{calrsfs} 
\DeclareMathAlphabet{\pazocal}{OMS}{zplm}{m}{n}

\usepackage[displaymath,textmath,graphics, subfigure, floats]{preview} 
\PreviewEnvironment{center} 
\PreviewEnvironment{pspicture} 

\makeatother

\title{The Vector Balancing Constant for Zonotopes}
\author{Laurel Heck\thanks{University of Washington, Seattle. Email: lheck98@uw.edu. Supported by an NSF Graduate
Research Fellowship.} \and Victor Reis\thanks{University of Washington, Seattle. Email: voreis@uw.edu} \and Thomas Rothvoss\thanks{University of Washington, Seattle. Email: rothvoss@uw.edu. Supported by NSF CAREER grant 1651861 and a David \& Lucile Packard Foundation Fellowship.}}
\date{}

\begin{document}

\maketitle


\begin{abstract}
  The vector balancing constant $\vb(K,Q)$ of two symmetric convex bodies $K,Q$ is the minimum  $r \geq 0$
  so that any number of vectors from $K$ can be balanced into an $r$-scaling of $Q$.
  A question raised by Schechtman is whether for any zonotope $K \subseteq \setR^d$ one has $\textrm{vb}(K,K) \lesssim \sqrt{d}$. Intuitively, this asks whether a natural geometric generalization of Spencer's Theorem (for which $K = B^d_\infty$) holds.
  We prove that for any zonotope $K \subseteq \setR^d$ one has $\vb(K,K) \lesssim \sqrt{d} \log \log \log d$.
  Our main technical contribution is a tight lower bound on the Gaussian measure of any section of a normalized zonotope, generalizing Vaaler's Theorem for cubes.
  We also prove that for two different normalized zonotopes $K$ and $Q$ one has $\vb(K,Q) \lesssim \sqrt{d \log d}$.
  All the bounds are constructive and
  the corresponding colorings can be computed in polynomial time.
\end{abstract}

\section{Introduction}

\emph{Discrepancy theory} is a subfield of combinatorics where one is given a set system $(X,\pazocal{F})$ with a ground set $X$ and a family of sets $\pazocal{F} \subseteq 2^X$, and the goal is to find the coloring that minimizes
the maximum imbalance, i.e.
\[
  \disc(\pazocal{F}) = \min_{x \in \{ -1,1\}^X} \max_{S \in \pazocal{F}} \Big|\sum_{j \in S} x_j\Big|.
\]
A slightly more general linear-algebraic view is that one is given a matrix $A \in [-1,1]^{d \times n}$ and its discrepancy
is defined as $\min_{x \in \{ -1,1\}^n} \|Ax\|_{\infty}$. The best known result in this area is certainly Spencer's Theorem~\cite{SixStandardDeviationsSuffice-Spencer1985} which states that for any $n \leq d$ one has $\textrm{disc}(A) \leq O(\sqrt{n \log (\frac{2d}{n})})$. The challenging aspect of that Theorem is that --- say for $n=d$ --- a uniform random coloring $x \sim \{ -1,1\}^n$ will only give a $\Theta(\sqrt{n \log n})$ bound. Instead, Spencer~\cite{SixStandardDeviationsSuffice-Spencer1985} applied the \emph{partial coloring method} which had been first used by
Beck~\cite{Beck-RothsEstimateIsSharp1981}.

The original proofs of the partial coloring method are based on the \emph{pigeonhole principle} and are
non-constructive. The first polynomial time algorithm to actually find the coloring guaranteed by
Spencer~\cite{SixStandardDeviationsSuffice-Spencer1985} is due to Bansal~\cite{ConstructiveDiscMin-Bansal-FOCS2010}, followed by a sequence of algorithms~\cite{ConstructiveDiscMin-LovettMeka-FOCS2012,ConstructiveDiscrepancy-Rothvoss-FOCS2014,DBLP:journals/corr/LevyRR16,DBLP:journals/rsa/EldanS18} that either work in more general settings or are simpler.

Discrepancy theory is an extensively studied topic with many applications in mathematics and computer science.
To give two concrete examples, Nikolov, Talwar and Zhang~\cite{DBLP:conf/stoc/NikolovTZ13} showed a connection between differential privacy and
hereditary discrepancy, and the best known approximation algorithm for Bin Packing uses a discrepancy-based rounding~\cite{DBLP:conf/soda/HobergR17}.
Other applications can be found in data structure lower bounds, communication complexity and pseudorandomness; we refer to the book of Chazelle~\cite{DiscrepancyMethod-Chazelle-2000} for a more detailed account. The seminal result of Batson, Spielman and Srivastava~\cite{TwiceRamanujanSparsifiers-BatsonSpielmanSrivastava-STOC09} on the existence of linear-size spectral sparsifiers for graphs can also be interpreted as a discrepancy-theoretic result, see \cite{DBLP:conf/soda/ReisR20} for details.

For the purpose of this paper, it will be convenient to introduce more general notation.
For two symmetric convex bodies $K,Q \subseteq \setR^d$ we define the \emph{vector balancing constant} $\textrm{vb}(K,Q)$ as the smallest number $r \ge 0$ so that for any vectors $u_1,\ldots,u_n \in K$
one can find signs $x \in \{-1,1\}^n$ so that the signed sum $x_1 u_1 + \cdots + x_n u_n$ is in $rQ$.
We also denote $\textrm{vb}_n(K,Q)$ as the same quantity where we fix the number of vectors to be $n$. 
For example, Spencer's Theorem~\cite{SixStandardDeviationsSuffice-Spencer1985} can then be rephrased as $\vb(B_{\infty}^d,B_{\infty}^d) = \Theta(\sqrt{d})$ and as $\vb_n(B_{\infty}^d,B_{\infty}^d) = \Theta(\sqrt{n \log (\frac{2d}{n})})$ for $n \leq d$.
Here we denote $B_p^{d}$ as the $d$-dimensional unit ball of the norm $\| \cdot \|_p$.
Moreover for a Euclidean ball one can easily prove that  $\textrm{vb}(B_2^d,B_2^d) = \Theta(\sqrt{d})$ and for the $\ell_1$-ball we have
$\vb(B_1^d,B_1^d) = \Theta(d)$.

While Spencer's Theorem itself is tight, at least three candidate generalizations have been suggested in
the literature --- all three are unsolved so far.

\paragraph{The Beck-Fiala Conjecture.} Suppose we have a set system $(X,\pazocal{F})$ in which every element is
in at most $t$ sets. Beck and Fiala~\cite{IntegerMakingTheorems-BeckFiala81} proved  using a
linear-algebraic argument that in this case the discrepancy is bounded by $2t$ and they state
the conjecture that the correct dependence should be $O(\sqrt{t})$. The same proof of~\cite{IntegerMakingTheorems-BeckFiala81} also shows that
$\vb(B_1^d,B_{\infty}^d) \leq 2$.
However, the Beck-Fiala Conjecture is wide open and the best known bounds are $O(\sqrt{t \log n})$~\cite{Banaszczyk-RSA1998, GramSchmidtWalk-BansalDGL-STOC2018} and $2t - \log^*(t)$~\cite{bukh_2016}.
In fact, Koml\'os Conjecture of $\vb(B_2^d,B_{\infty}^d) \leq O(1)$ is even more general; here the
best known bound is $\vb(B_2^d, B_{\infty}^d) \leq O(\sqrt{\log(d)})$~\cite{Banaszczyk-RSA1998}.

\paragraph{The Matrix Spencer Conjecture.} A conjecture popularized by Zouzias~\cite{Zouzias2012}
and Meka~\cite{MekaBlog2014} claims that for any symmetric matrices $A_1,\ldots,A_n \in \setR^{n \times n}$
with all eigenvalues in $[-1,1]$, there are signs $x \in \{ -1,1\}^n$ so that the maximum singular value
of $\sum_{i=1}^n x_iA_i$ is at most $O(\sqrt{n})$. Using standard matrix concentration bounds, one can prove that a random coloring attains a value of at most $O(\sqrt{n \log n})$.
Moreover, one can prove the conjectured upper bound of $O(\sqrt{n})$ under the additional assumption that the matrices are block-diagonal with constant size blocks~\cite{DBLP:conf/stoc/DadushJR22}, or have rank $O(\sqrt{n})$~\cite{MatrixDiscViaQuantum-STOC2022}. Based on recent progress on matrix concentration, it is possible to obtain the same under the weaker condition that they have rank at most $\frac{n}{\log^3(n)}$~\cite{BansalJiangMeka-MatrixSpencerUpToPolylog-2022-Arxiv}.

\paragraph{The vector balancing constant of zonotopes.}

A \emph{zonotope} is defined as the linear image of a cube. If $A \in \setR^{m \times d}$ is a matrix with $m \ge d$, we can write a $d$-dimensional zonotope in the form $K = \{ \sum_{i=1}^m y_iA_i \mid y \in [-1,1]^m \} = A^{\top} B_{\infty}^m \subseteq \setR^d$. Note that $m$ is the  \emph{number of segments} of the zonotope. The cube $B_{\infty}^d$ is trivially a zonotope, and it is known that for every $p \geq 2$, the ball $B_{p}^n$ is the limit of a sequence of zonotopes, called a \emph{zonoid}~\cite{BLM89}.
Schechtman~\cite{AIMWorkshop2007} raised the question whether it is true that for any zonotope $K \subseteq \setR^d$ one has  $\vb(K,K) \lesssim \sqrt{d}$ where we write $A \lesssim B$ if $A \leq C \cdot B$ for a universal constant $C>0$.
The best known bound of $\vb(K,K) \lesssim \sqrt{d \log \log d}$ is a direct consequence of Spencer's theorem and the fact that zonotopes can be \emph{sparsified} up to a constant factor with only $O(d \log d)$ segments~\cite{EmbeddingL1-Talagrand-PAMS1990}. An affirmative answer to Schechtman's question would follow from an $O(d)$ bound, or equivalently whether an $\ell_1$-analogue of~\cite{TwiceRamanujanSparsifiers-BatsonSpielmanSrivastava-STOC09} is true. We defer to Section~\ref{sec:OpenProblems} for details. 

\subsection{Our contributions}
Our main result is an almost-proof of Schechtman's conjecture (falling short only by a $\log \log \log d$ term).
\begin{theorem} \label{thm:VBofKwithK}
  For any zonotope $K \subseteq \setR^d$ one has $\textrm{vb}(K,K) \lesssim \sqrt{d} \log \log \log d$. Moreover, for any
  $v_1,\ldots,v_n \in K$ one can find in randomized polynomial time a coloring $x \in \{ -1,1\}^n$ with $\|\sum_{i=1}^n x_iv_i\|_K \lesssim \sqrt{d} \log \log \log d$.
\end{theorem}
The claim is invariant under linear transformations to $K$ and so it will be useful to place $K$ in a normalized position.
For this sake, we make the following definition:
\begin{definition}
  A matrix $A \in \setR^{m \times d}$ is called \emph{approximately regular} if the following holds:
  \begin{enumerate*}
  \item[(i)] The columns $A^{1},\ldots,A^d$ are orthonormal.
  \item[(ii)] The rows satisfy  $\|A_i\|_2 \leq 2 \sqrt{\frac{d}{m}}$ for all $i=1,\ldots,m$.
\end{enumerate*}
\end{definition}
Then we call a zonotope $K \subseteq \setR^d$ \emph{normalized} if there exists a matrix $A \in \setR^{m \times d}$ that is approximately
regular so that $K = \sqrt{\frac{d}{m}} A^\top B_{\infty}^m$. We choose the scaling so that any cube $B_{\infty}^d$ is indeed normalized
and zonotopes with any number of segments are comparable to $B_{\infty}^d$ in terms of volume and radius. 

Our main technical contribution is a tight lower bound for the Gaussian measure of sections of any normalized zonotope. 
\begin{theorem} \label{thm:SlicesOfNormZonotopes}
  For any normalized zonotope $K \subseteq \setR^d$, any subspace $H \subseteq \setR^d$ with $n := \dim(H)$
 and any $t \geq 1$, one has $\gamma_H(t \cdot C \cdot K \cap H) \geq \exp(-e^{-t^2/2} \cdot n)$ where $C>0$ is a universal constant.
\end{theorem}
In order to prove Theorem~\ref{thm:SlicesOfNormZonotopes}, we show that a normalized zonotope can be decomposed into
$\Theta(\frac{m}{d})$ many smaller zonotopes with $\Theta(d)$ many segments each. This decomposition requires an iterative application
of the Kadison-Singer theorem by Marcus, Spielman and Srivastava~\cite{KadisonSingerProblem-MSS-AnnalsOfMath2015}.
Then we prove the statement of Theorem~\ref{thm:SlicesOfNormZonotopes} for such simpler zonotopes and derive the lower bound
on $\gamma_H(t \cdot C \cdot K \cap H)$ by using log-concavity of the Gaussian measure.

We can also use Theorem~\ref{thm:SlicesOfNormZonotopes} to show how to balance vectors between different normalized zonotopes:
\begin{theorem} \label{thm:VBofKwithQ}
  For any normalized zonotopes $K,Q \subseteq \setR^d$ one has $\vb(K,Q) \lesssim \sqrt{d \log d}$.
  Moreover, for any $v_1,\ldots,v_n \in K$ one can find in randomized polynomial time a coloring $x \in \{ -1,1\}^n$ such that ${\|\sum_{i=1}^n x_iv_i\|_Q \lesssim \sqrt{d \cdot \log \min\{ d,n\}}}$ .
\end{theorem}

\section{Preliminaries}

We review a few facts that we rely on later.
\paragraph{Probability.}
By $\gamma_n$ we denote the \emph{(standard) Gaussian density}  $\frac{1}{(2\pi)^{n/2}}e^{-\|x\|_2^2/2}$.
For the corresponding distribution we will write $N(0,I_n)$. For a subspace $F \subseteq \setR^n$
we write $I_F \in \setR^{n \times n}$ as the identity on the subspace; in particular $I_F = \sum_{i=1}^{\dim(F)} u_iu_i^T$ where $u_1,\ldots,u_{\dim(F)}$ is any orthonormal basis of $F$. 
A \emph{strip} is a symmetric convex body of the form $P = \{ x \in \setR^n : |\left<a, x\right>| \leq 1\}$
with $a \in \setR^n$. 
\begin{theorem}[\v{S}id\'ak-Khatri] \label{lem:SidakKhatri}
  For any two symmetric convex bodies $P,Q \subseteq \setR^n$ where at least one is a strip,  one has $\gamma_n(P \cap Q) \geq \gamma_n(P) \cdot \gamma_n(Q)$.
\end{theorem}
More recently, Royen~\cite{Royen2014Arxiv} proved that this is indeed true for any pair of symmetric convex bodies, but the weaker result suffices for us.
\begin{lemma} \label{lem:ProbOfSubspaceVsBody}
  For any symmetric convex body $K$ and any subspace $H \subseteq \setR^n$ one has $\gamma_H(K \cap H) \geq \gamma_n(K)$.
\end{lemma}
We will use the following convenient estimate on the Gaussian measure of a strip:
\begin{lemma} \label{lem:GaussianMeasureOfStrip}
  For any $a \in \setR^n$ with $\|a\|_2 \leq 1$ and $t \geq 1$ one has
  \[\Pr_{y \sim N(0,I_n)}[|\left<a,y\right>| \leq t] \geq \exp(-e^{-t^2/2} \cdot \|a\|_2^2).\]
\end{lemma}

The following comparison inequality (see e.g. Ledoux and Talagrand~\cite{LedouxTalagrandBook2011}) will also be useful:

\begin{lemma}\label{lem:ComparisonGaussians}
  Let $K$ be a symmetric convex body and let $0 \preceq A \preceq B$. Then
  \[\Pr_{y \sim N(0,A)}[y \in K] \ge \Pr_{y \sim N(0,B)}[y \in K].\]
\end{lemma}

We prove these lemmas in Appendix B. The following lemma allows us to dismiss constant scaling factors, see ~\cite{tkocz2015high}:

\begin{lemma} \label{lem:ScalingLemma}
 Let $K \subset \setR^n$ be a measurable set and $B$ be an Euclidean ball centered at the origin such that $\gamma_n(K) = \gamma_n(B)$. Then $\gamma_n(tK) \ge \gamma_n(tB)$ for all $t \in [0, 1]$. In particular, if $\gamma_n(C_1 \cdot K) \ge e^{-C_1 n}$ for some constant $C_1 \geq 1$ then also $\gamma_n(K) \ge e^{-C_2 n}$ for some $C_2 := C_2(C_1) >0$.  
\end{lemma}

\paragraph{Discrepancy theory.}
First we give a full statement of Spencer's theorem that we mentioned earlier:
\begin{theorem}[Spencer's Theorem~\cite{SixStandardDeviationsSuffice-Spencer1985,ConstructiveDiscMin-LovettMeka-FOCS2012}] \label{thm:spencer}
  For any $A \in [-1,1]^{m \times n}$ with $m \geq n$ there are polynomial time computable signs $x \in \{ -1,1\}^n$ so that
  $\|Ax\|_{\infty} \lesssim  \sqrt{n \log(\frac{2m}{n})}$. More generally, for any shift $x_0 \in [-1,1]^n$, there is a polynomial time computable $x \in \setR^n$ so that $x+x_0 \in \{-1,1\}^n$ and $\|A(x + x_0)\|_\infty \lesssim \sqrt{n \log(\frac{2m}{n})}$.
\end{theorem}
To be exact, the first algorithm giving a bound of $O(\sqrt{n} \log(\frac{2m}{n}))$ is due to Bansal~\cite{ConstructiveDiscMin-Bansal-FOCS2010} and the tight algorithmic bound is due to Lovett and Meka~\cite{ConstructiveDiscMin-LovettMeka-FOCS2012}.  

 We say that a vector $x \in \setR^n$ is a \emph{good partial coloring} if $x \in [-1,1]^n$
 with $|\{ j \in [n] : x_j \in \{ -1,1\} \}| \geq n/2$. We will need a connection between good partial colorings and
 Gaussian measure lower bounds.

\begin{theorem}[\cite{RR22}, special case of Theorem 6] \label{thm:partial_coloring}
For any $\alpha > 0$, there is a constant $c:= c(\alpha) > 0$ and a randomized polynomial time algorithm that for a symmetric convex body $K \subseteq \mathbb{R}^n$, a $2n/3$-dimensional subspace $F \subseteq \setR^n$ with $\gamma_F(K \cap F) \geq e^{-\alpha n}$ and a shift $y \in (-1,1)^n$, finds $x \in c\cdot K \cap F$ so that $x+y$ is a good partial coloring.
\end{theorem}

We will also need a theorem of Banaszczyk~\cite{Banaszczyk-RSA1998} (whose algorithmic version is due to~\cite{GramSchmidtWalk-BansalDGL-STOC2018}).
\begin{theorem}[Banaszczyk's Theorem]\label{thm:Banaszczyk}
  Let $K \subseteq \setR^{d}$ be a convex set with $\gamma_d(K) \geq \frac{1}{2}$ and let $v_1,\ldots,v_n \in B_2^d$.
  Then there is a randomized polynomial time algorithm to compute signs $x \in \{ -1,1\}^n$ so that $\sum_{j=1}^n x_jv_j \in CK$
  where $C>0$ is a universal constant.
\end{theorem}
For many decades, the \emph{Kadison-Singer problem} was an open question in operator theory. It was 
finally resolved in 2015:
\begin{theorem}[Marcus, Spielman, Srivastava~\cite{KadisonSingerProblem-MSS-AnnalsOfMath2015}] \label{thm:MSS}
  Let $v_1,\ldots,v_m \in \setR^n$ so that $\sum_{i=1}^m v_iv_i^\top = I_d$ and let $\varepsilon>0$ so that $\|v_i\|_2^2 \leq \varepsilon$ for all $i \in [m]$.
  Then there is a partition $[m] = S_1 \dot{\cup} S_2$ so that for both $j \in \{ 1,2\}$ one has 
  \[
    \Big\|\sum_{i \in S_j} v_iv_i^\top -\frac{1}{2}I_d\Big\|_{\mathrm{op}} \leq 3\sqrt{\varepsilon} 
  \]
\end{theorem}
In the definition of $\textrm{vb}(K,Q)$, there is no upper bound on the number of vectors to be balanced.
But it is well-known that up to a constant factor, the worst-case is attained for $d$ many vectors. Let 
\[
 \textrm{vb}_n(K,Q) := \inf\Big\{ r \geq 0 \mid \forall u_1,\ldots,u_n \in K : \exists x \in \{ -1,1\}^n : \sum_{j=1}^n x_jv_j \in rQ\Big\} 
\]
be the vector balancing variant with $n$ vectors, so that $\vb(K,Q) := \sup_{n \in \setN} \vb_n (K,Q).$
\begin{theorem}[\cite{LovaszSpencerVesztergombi-1986}] \label{thm:reduction_nd}
For any symmetric convex $K,Q \subseteq \setR^d$, $\mathrm{vb}(K,Q) \leq 2 \cdot \mathrm{vb}_{d}(K,Q)$.
\end{theorem}
The reduction underlying the inequality is algorithmic as well.
\paragraph{Zonotopes.} A substantial amount of work in the literature has been done on the question of how one can sparsify an arbitrary zonotope with another zonotope that has fewer segments, while losing only a constant factor approximation. The first bound of $O(d^2)$~\cite{Schechtman1987} was improved to $O(d \log^3 d)$~\cite{BLM89}. We highlight the current best known bound: 
\begin{theorem}[Talagrand~\cite{EmbeddingL1-Talagrand-PAMS1990}] \label{thm:sparsify}
For any zonotope $K \subseteq \setR^d$ and $0<\varepsilon \leq \frac{1}{2}$, there is a zonotope $Q$ with at most $O(\frac{d}{\varepsilon^2} \log d)$ segments so that $Q \subseteq K \subseteq (1+\varepsilon)Q$.
\end{theorem}

We refer to the approach of Cohen and Peng~\cite{LewisWeights} for an elementary exposition of the $O(d \log d)$ bound.

Finally, we justify why it suffices to consider normalized zonotopes:
\begin{lemma} \label{lem:normalized}
For any full-dimensional zonotope $K = A^\top B^m_\infty \subseteq \setR^d$, there is a normalized zonotope $\tilde{K}$ and an invertible linear map $T$ so that $\frac{4}{5} \tilde{K} \subseteq T(K) \subseteq \tilde{K}$. In particular, $\frac{4}{5}\vb(\tilde{K},\tilde{K}) \leq \vb(K,K) \leq \frac{5}{4}\vb(\tilde{K},\tilde{K})$.
\end{lemma}
We show the argument in Appendix~\ref{sec:NormalizingZonotopes}.

\begin{lemma} \label{lem:contained_ball}
Any normalized zonotope $K \subseteq \setR^d$ satisfies $K \subseteq \sqrt{d} B_2^d$.
\end{lemma}
\begin{proof}
  We write $K = \sqrt{\frac{d}{m}} A^\top B_{\infty}^{m}$ where $A \in \setR^{m \times d}$. Note that $A^\top A = I_d$
  by orthonormality of the columns of $A$ and so $\|A\|_{\textrm{op}} = \|A^\top A\|_{\textrm{op}}^{1/2} = 1$.
  By definition, for any $x \in K$ there is a $y \in B_{\infty}^{m}$ with $x = \sqrt{\frac{d}{m}}A^\top y$, so that \[\|x\|_2 = \sqrt{\frac{d}{m}} \|A^\top y\|_2 \leq \sqrt{\frac{d}{m}} \|A^\top\|_{\textrm{op}} \cdot \|y\|_2 \leq \sqrt{d}.\qedhere \]
\end{proof}

\section{Sections of normalized zonotopes}
In this section we prove Theorem~\ref{thm:SlicesOfNormZonotopes}, showing that all sections of zonotopes
are large. To be be more precise, we prove the following more general measure lower bound:
\begin{theorem}
  For any normalized zonotope $K \subseteq \setR^d$, any subspace $H \subseteq \setR^d$ with $n := \dim(H)$
 and any $t \geq 1$, one has $\gamma_H(t \cdot C \cdot K \cap H) \geq \exp(-e^{-t^2/2} \cdot n)$ where $C>0$ is a universal constant.
\end{theorem}
In the most basic form where $K = B_{\infty}^d$ is a cube and $t=1$, the statement is similar to
a result of Vaaler~\cite{Vaaler1979} who proved that $\Vol_H(K \cap H) \geq 2^{n}$ for any $n$-dimensional subspace $H \subseteq \setR^d$; though the geometry of a zonotope is more complex and the proof strategy is rather different.

\subsection{A first direct lower bound}

We begin with a simple estimate on the Gaussian measure of the section of a zonotope
where we drop the scalar of $\sqrt{\frac{d}{m}}$.
Hence this bound will be tight if the number of segments is close to $d$ but rather loose otherwise.
We denote $\Pi_H$ as the orthogonal projection into a subspace $H$.
\begin{lemma} \label{lem:MeasureOfSliceDirectProof}
  Let $K := A^{\top}B_{\infty}^m \subseteq \setR^d$ be a zonotope where $A \in \setR^{m \times d}$ is a matrix with orthonormal columns. 
  Then for any subspace $H \subseteq \setR^d$  with $n := \dim(H)$ and any $t \geq 1$ one has $\gamma_H(t \cdot K \cap H) \geq \exp(-e^{-t^2/2} \cdot n)$.
\end{lemma}

\begin{proof}
  Let $U \in \setR^{d \times n}$ be a matrix with orthonormal columns $U^1, \dots, U^n$ spanning $H$. Then if
  we draw $y \sim N(0,I_n)$, $Uy$ is indeed a standard Gaussian in the subspace $H$.
  By assumption, $\sum_{i=1}^m A_iA_i^{\top} = I_d$, and this can be used to write any outcome of
  the random process as
  \begin{equation}
    Uy = \sum_{j=1}^n y_j I_dU^j 
    = \sum_{i=1}^m A_i \sum_{j=1}^n y_j\left<A_i,U^j\right>=\sum_{i=1}^m A_i\left<y,U^{\top}A_i\right>.  \label{eq:ExprForUy}
\end{equation}
Here one should think of $U^{\top}A_i \in \setR^n$ as the coordinates of $\Pi_H(A_i)$ in terms of the basis $U$ of $H$.
  From the expression in \eqref{eq:ExprForUy} we can draw the following conclusion: \\
  {\bf Claim I.} \emph{For any $y \in \setR^n$ and $s>0$ one has $(|\left<y,U^{\top}A_i\right>| \leq s \; \forall i \in [m]) \Rightarrow Uy \in sK$.} \\
  Then Claim I gives a simple sufficient (but in general not necessary) condition for $Uy$ to lie in the zonotope $K$.
  Next, we can see that
  \[
  \sum_{i=1}^m \|U^{\top}A_i\|_2^2 = \sum_{i=1}^m \textrm{Tr}\big[UU^{\top}A_iA_i^{\top}\big] = \textrm{Tr}[UU^{\top}] = n
  \] 
Then we can use Claim I and the inequality of \v{S}id\'ak-Khatri to lower bound the Gaussian measure by
\begin{eqnarray*}
  \gamma_H(t \cdot K \cap H) &=& \Pr_{y \sim N(0,I_n)}[Uy \in t \cdot K] \\ &\geq& \Pr_{y \sim N(0,I_n)}\big[|\left<U^{\top}A_i,y\right>| \leq t \; \; \forall i \in [m]\big] \\
                     &\stackrel{\textrm{Lem~}\ref{lem:SidakKhatri}}{\geq}&  \prod_{i=1}^m \Pr_{y \sim N(0,I_n)}\big[ |\left<U^{\top}A_i,y\right> | \leq t\big]  \\
                    \\ &\stackrel{\textrm{Lem~\ref{lem:GaussianMeasureOfStrip}}}{\geq}& \prod_{i=1}^m \exp\big(-e^{-t^2/2}\|U^{\top}A_i\|_2^2\big) \\ &=& \exp\Big(-e^{-t^2/2}\sum_{i=1}^m \|U^{\top}A_i\|_2^2\Big) = \exp\big(-e^{t^2/2}n\big)
\end{eqnarray*}
Here we have used that $\|U^{\top}A_i\|_2 \leq \|A_i\|_2 \leq 1$ which follows by the orthonormality of the columns of $A$.
\end{proof}

It is somewhat unfortunate that Claim I shown above requires that $\sum_{i=1}^m A_iA_i^{\top}$ is \emph{exactly} the identity and an approximation is not enough. But we can fix this by a rescaling argument:

\begin{lemma} \label{lem:MeasureOfSliceDirectProofII}
  Let $K = A^{\top}B_{\infty}^m \subseteq \setR^d$ be a zonotope where $A \in \setR^{m \times d}$ is a matrix so that $\sum_{i=1}^m A_iA_i^\top \succeq \alpha I_d$ for some $\alpha>0$.
  Then for any $n$-dimensional subspace $H \subseteq \setR^d$ and any $t \geq 1$ one has $\gamma_H(\frac{t}{\sqrt{\alpha}} \cdot K \cap H) \geq \exp\big(-e^{-t^2/2} \cdot n\big)$.
\end{lemma}
\begin{proof}
  Scaling $K$ by $\frac{1}{\sqrt{\alpha}}$ is equivalent to scaling $\sum_{i=1}^m A_iA_i^\top$ by $\frac{1}{\alpha}$,
  hence we may assume that indeed $\alpha = 1$.
Abbreviate $M := \sum_{i=1}^m A_iA_i^\top \succeq I_d$ which is a symmetric positive definite matrix.
  Consider the matrix $\tilde{A} \in \setR^{m \times d}$ with rescaled rows
  $\tilde{A}_{i} := M^{-1/2}A_{i}$, so that $\sum_{i=1}^m \tilde{A}_i\tilde{A}_i^\top = I_d$.
  Let $\tilde{K} := \tilde{A}^\top B_{\infty}^{m} = M^{-1/2}(K)$ and $\tilde{H} := M^{-1/2}(H)$ be the rescaled zonotope
  and subspace.
  Let $U^{1},\ldots,U^n$ be an orthonormal basis of $H$. Then with $\tilde{U} = M^{-1/2}U$, 
   $\tilde{U}^1,\ldots,\tilde{U}^n$ will be the basis of $\tilde{H}$, but it will not be orthogonal in general.
  However, for $y \sim N(0,I_n)$ one has $\textrm{Cov}(\tilde{U}y) = \tilde{U}\tilde{U}^{\top} = M^{-1/2}UU^{\top}M^{-1/2} \preceq I_{\tilde{H}}$. Then
  \[
  \Pr_{y \sim N(0,I_d)}[Uy \in tK] = \Pr_{y \sim N(0,I_d)}[\tilde{U}y \in t \tilde{K}] \stackrel{\textrm{Lem~\ref{lem:ComparisonGaussians}}}{\geq} \Pr_{y \sim N(0,I_{\tilde{H}})}[y \in t \tilde{K}] \stackrel{\textrm{Lem~\ref{lem:MeasureOfSliceDirectProof}}}{\geq} \exp\big(-e^{-t^2/2}n\big). \qedhere
  \]
\end{proof}


\subsection{Decomposition of normalized zonotopes}

The next step in our proof strategy is to decompose the rows of an approximately regular matrix $A \in \setR^{m \times d}$
into $\Theta(\frac{m}{d})$ many blocks $J \subseteq [m]$ so that $\sum_{i \in J} A_iA_i^{\top} \succeq \Omega(\frac{d}{m}) \cdot I_d$.
For this purpose, we formulate a slight variant of Theorem~\ref{thm:MSS}.
\begin{lemma}\label{lem:MSS-only-LB}
  Let $v_1,\ldots,v_m \in \setR^d$ be vectors with $\sum_{i=1}^m v_iv_i^{\top} \succeq L\cdot I_d$ for some $L>0$ and let $\varepsilon := \max_{i=1,\ldots,m} \|v_i\|_2^2$.
  Then there is a partition $[m] = S_1 \dot{\cup} S_2$ so that
  \[
  \sum_{i \in S_j} v_iv_i^{\top} \succeq \Big(\frac{L}{2}-3\sqrt{L\varepsilon}\Big) I_d \quad \forall j \in \{ 1,2\}
\]
\end{lemma}
\begin{proof}
  Abbreviate $M := \sum_{i=1}^m v_iv_i^{\top}$ which is a PSD matrix with $M \succeq L \cdot I_d$.
  Define $v_i' := M^{-1/2}v_i$. Then
  $\sum_{i=1}^m v_i'(v_i')^{\top} = M^{-1/2}\big(\sum_{i=1}^m v_iv_i^{\top}\Big)M^{-1/2} = I_d$.
We set $\varepsilon' := \frac{\varepsilon}{L}$ and verify that for all $i$ one has  
  $
    \|v_i'\|_2^2 = v_i^\top M^{-1} v_i \le v_i^\top (\tfrac{1}{L} I_d) v_i = \frac{\|v_{i}\|_2^2}{L} \leq  \varepsilon'.
  $
Then we apply Theorem~\ref{thm:MSS} to the vectors $\{v_i'\}_{i \in [m]}$ and obtain 
a partition $[m] = S_1 \dot{\cup} S_2$ so that for $j \in \{ 1,2\}$ one has 
\[
  \color{black} {M^{-1/2}\Big(\sum_{i \in S_j} v_iv_i^{\top}\Big)M^{-1/2} 
                        =  \sum_{i \in S_j} v_i'(v_i')^\top \stackrel{\textrm{Thm~\ref{thm:MSS}}}{\succeq} \Big(\frac{1}{2}-3\sqrt{\varepsilon/L}\Big) I_d},
                      \]
and using the fact that $A \succeq B \implies M^{1/2} A M^{1/2} \succeq M^{1/2} B M^{1/2}$, we conclude
\[\sum_{i \in S_j} v_iv_i^{\top} \succeq \Big(\frac{1}{2}-3\sqrt{\varepsilon/L}\Big) M^{1/2} I_d M^{1/2} \succeq \Big(\frac{L}{2}-3\sqrt{L\varepsilon}\Big) I_d. \qedhere \]
\end{proof}
Now to the main lemma of this section where we decompose an approximately regular matrix by iteratively applying Lemma~\ref{lem:MSS-only-LB}.
\begin{lemma} \label{lem:FindingkDisjSubmatricesOfA}
  There is a universal constant $C>0$ so that the following holds. Let $A \in \setR^{m \times d}$ be an approximately regular matrix. Then there are disjoint subsets $J_1 \dot{\cup} \cdots \dot{\cup} J_k \subseteq [m]$ with $k \geq \frac{m}{Cd}$ and $|J_{\ell}| \leq Cd$ and $\sum_{i \in J_{\ell}} A_iA_i^\top \succeq \frac{1}{Ck} I_d$ for all $\ell \in [k]$.
\end{lemma}
\begin{proof}
  If $\frac{m}{d} \le C$ we may set $k = 1$ and $J_1 = [m]$, so assume $m \ge Cd$. Set $\varepsilon := 4\frac{d}{m}$ so that $\|A_i\|_2^2 \leq \varepsilon$ for all $i \in [m]$.
  Let $t \in \setN$ be a parameter that we choose later. For $s \in \{ 0,\ldots,t\}$ we will obtain
  partitions $\pazocal{P}_s$ of the row indices starting with $\pazocal{P}_0 := \{ [m] \}$ so that $\pazocal{P}_{s+1}$ is a
  refinement of $\pazocal{P}_s$ and moreover $|\pazocal{P}_s| = 2^s$.
  More precisely, in each iteration $s \in \{ 0,\ldots,t-1\}$ and for each $S \in \pazocal{P}_s$,
  we apply Lemma~\ref{lem:MSS-only-LB} to the vectors $\{A_i\}_{i \in S}$; if $S = S_1 \dot{\cup} S_2$ is the obtained partition,
  then we add $\{S_1,S_2\}$ to $\pazocal{P}_{s+1}$.
We first analyze the corresponding eigenvalue lower bound.   
Define $L_s := 2^{-s} - 15\sqrt{2^{-s}\varepsilon}$. \\ 
 {\bf Claim.} \emph{If $2^t \leq \frac{m}{Cd}$ for a large enough constant $C>0$, then for all $s \in \{ 0,\ldots,t\}$ one has $\sum_{i \in S} A_iA_i^{\top} \succeq L_s I_d$ for all $S \in \pazocal{P}_s$.} \\
 {\bf Proof of Claim.} Clearly  $L_s \leq 2^{-s}$ all $s \geq 0$. 
 We will prove the claim by induction on $s$. For $s=0$ one has $\pazocal{P}_0 = \{ [m]\}$ and the claim is true as $L_0 \leq 1$.
  Now consider an iteration $s \in \{ 0,\ldots,t-1\}$ and suppose $S \in \pazocal{P}_s$ is split into $S= S_1 \dot{\cup} S_2$. Then $\sum_{i \in S_j} A_iA_i^\top \succeq (\frac{L_s}{2} - 3\sqrt{L_s\varepsilon}) I_d$ for both $j \in \{ 1,2\}$. This is at least $L_{s+1}$ as:
  \[
    \frac{L_s}{2} - 3\sqrt{L_s\varepsilon} \stackrel{L_s\leq 2^{-s}}{\geq} \frac{L_s}{2} - 3\sqrt{2^{-s}\varepsilon}
     \ge 2^{-(s+1)} - \frac{15}{2}\sqrt{2^{-s}\varepsilon} - 3\sqrt{2^{-s}\varepsilon} \geq 2^{-(s+1)} - 15\sqrt{2^{-(s+1)}\varepsilon}.
  \]
  Here we use $15/2 +3 \leq 15\sqrt{2^{-1}}$. This shows the claim. \qed

  For a large enough constant $C$, we pick $t \in \setN$ so that $\frac{m}{2Cd} \leq 2^t \leq \frac{m}{Cd}$.
  Then $L_t \geq \frac{Cd}{m} - 15 \sqrt{\frac{2Cd}{m} \cdot 4\frac{d}{m} } = \frac{d}{m} \cdot (C - 15 \sqrt{8C}) \geq \frac{C}{2} \cdot \frac{d}{m}$ for $C$ large enough.
  Moreover we know that $\E_{S \sim \pazocal{P}_t}[|S|] = \frac{m}{2^t} \leq 2Cd$.
  Then by Markov's inequality at least half the sets $S \in \pazocal{P}_t$ have at most $4Cd$ indices. Those sets will satisfy the statement.
\end{proof}

\subsection{Proof of Theorem~\ref{thm:SlicesOfNormZonotopes}}

Next we prove our main technical result, Theorem~\ref{thm:SlicesOfNormZonotopes}.
Recall that a measure $\mu$ on $\setR^d$ is called \emph{log-concave} if for all compact
subsets $S,T \subseteq \setR^d$ and $0 \le \lambda \le 1$ one has 
\[
 \mu(\lambda S + (1-\lambda)T) \geq \mu(S)^{\lambda} \cdot \mu(T)^{1-\lambda} 
\]
By induction one can verify that for any compact subsets $S_1, \dots, S_k \subseteq \setR^d$ and $\lambda_1,\ldots,\lambda_k \ge 0$ with $\sum_{i=1}^k \lambda_i = 1$ we have $\mu(\lambda_1 S_1 + \cdots + \lambda_k S_k) \geq \prod_{\ell=1}^k \mu(S_{\ell})^{\lambda_\ell}$.
Also recall that the Gaussian measure $\gamma_d$ is indeed log-concave, see e.g. \cite{AsymptoticGeometricAnalysisBook2005}. For a matrix $A \in \setR^{m \times d}$ and indices $J \subseteq [m]$
we denote $A_J \in \setR^{|J| \times d}$ as the submatrix of $A$ with rows in $J$. 

\begin{proof}[Proof of Theorem~\ref{thm:SlicesOfNormZonotopes}]
  Let $K \subseteq \setR^d$ be a normalized zonotope and let $H \subseteq \setR^d$ be a subspace with dimension $n$.
  Then we can write $K = \sqrt{\frac{d}{m}} A^{\top}B_{\infty}^m$ where $A \in \setR^{m \times d}$ is approximately regular.
  We use Lemma~\ref{lem:FindingkDisjSubmatricesOfA} to obtain disjoint subsets $J_1 \dot{\cup} \cdots \dot{\cup} J_k \subseteq [m]$ with $k \geq \frac{m}{Cd}$ so that $\sum_{i \in J_{\ell}} A_iA_i^{\top} \succeq \frac{d}{Cm} I_d$ where $C>0$ is a constant. Consider the zonotope $K_{\ell} := \sqrt{\frac{d}{m}} A_{J_{\ell}}^{\top}B_{\infty}^{|J_{\ell}|}$ generated by the rows with indices in $J_{\ell}$. Then we have $K_1 + \ldots + K_k \subseteq K$ and $(K_1 \cap H) + \ldots + (K_{k} \cap H) \subseteq K \cap H$.
  Note that for each $\ell \in [k]$ we have $k K_{\ell}  \supseteq \sqrt{\frac{k}{C}} A_{J_{\ell}}^{\top}B_{\infty}^{|J_{\ell}|}$, so that $\sum_{i \in J_{\ell}} (\sqrt{\tfrac{k}{C}} A_i)(\sqrt{\tfrac{k}{C}} A_i)^{\top} \succeq \frac{k}{C} \cdot \frac{d}{Cm} I_d\succeq \frac{1}{C^3} I_d$. 
  Then applying Lemma~\ref{lem:MeasureOfSliceDirectProofII} with $\alpha := \frac{1}{C^3}$ we have
  \[
    \gamma_H\big(t C^{3/2} k K_{\ell} \cap H\big) \geq \exp\big(-e^{t^2/2} \cdot n\big)
  \]
  for all $t \geq 1$.
  Finally, using log-concavity of the Gaussian measure we obtain
  \begin{eqnarray*}
  \gamma_H\big(tC^{3/2} K \cap H\big) &\geq& \gamma_H\big( (tC^{3/2}K_1 \cap H) + \ldots + (tC^{3/2}K_k \cap H)\big) \\ &\geq& \prod_{\ell=1}^{k} \gamma_H\big(tC^{3/2} \cdot k K_{\ell} \cap H\big)^{1/k} \geq \exp\big(-e^{-t^2/2} \cdot n\big). \qedhere
  \end{eqnarray*}
\end{proof}

\section{The vector balancing constant $\mathrm{vb}(K,K)$}

Next, we show how to translate measure lower bounds for sections into an improved bounds on the vector balancing constant.

\subsection{Tight partial colorings for zonotopes}

First we prove a generalization of the constant discrepancy partial coloring for the Koml\'os setting: 

\begin{lemma} \label{lem:partial_komlos_gen}
  Let $v_1,\ldots,v_n \in B_2^d$ and let $K \subseteq \setR^d$ be a symmetric convex body with $\gamma_H(K \cap H) \geq e^{-\alpha n}$ for some $\alpha > 0$
  where $H = \spn\{ v_1,\ldots,v_n\}$. Then there is a randomized polynomial time algorithm that given a shift $y \in (-1,1)^n$ finds a good
 partial coloring $x + y \in [-1,1]^n$ with $\sum_{j=1}^n x_jv_j \in c K$ where $c := c(\alpha)$ is a constant.
\end{lemma}

\begin{proof} 
  Let $Z \sim \sum_{j=1}^n z_j v_j$ where $z_i \sim N(0,1)$ are i.i.d. Gaussians so that $\mathbb{E}[ZZ^\top] = \sum_{j=1}^n v_j v_j^\top$ has trace $\tr \big[ \mathbb{E}[ZZ^\top]\big] = \sum_{j=1}^n \|v_j\|_2^2 \le n$. Let $u_1, \dots, u_r$ be an orthonormal basis of $H$ with $r \le n$, and write $\sum_{j=1}^n v_j v_j^\top = \sum_{j=1}^r \sigma_j u_j u_j^\top$. Since $\sum_{j=1}^n v_jv_j^{\top} \succeq 0$, we have $\sigma_j \geq 0$ for all $j$.
  Then after reindexing we may assume that  $0 \leq \sigma_1 \le \sigma_2 \le \dots \le \sigma_r$.
  Since $\sum_{j=1}^r \sigma_j = \sum_{j=1}^n \|v_j\|_2^2 \le n$ we know by Markov's Inequality that  $\sigma_{2n/3} \le 3/2$, denoting $\sigma_j = 0$ for $j > r$. Thus restricting to the subspaces $F := \spn\{u_1, \dots, u_{2n/3}\}$ and $V := \{ g \in \setR^n \mid \sum_{j=1}^n g_jv_j \in F\}$
  with $\dim(V) \geq \frac{2}{3}n$, we may lower bound
\begin{eqnarray*}
\Pr_{g \sim N(0, I_V)} \Big[\sum_{j=1}^n g_j v_j \in 3/2 \cdot K\Big] & = & \Pr_{g \sim N(0, I_{2n/3})} \Big[\sum_{j=1}^{2n/3} g_j \cdot \sigma_j u_j u_j^\top \in 3/2 \cdot K\Big] \\ & \stackrel{(*)}{\ge} & \Pr_{g \sim N(0, I_{2n/3})} \Big[\sum_{j=1}^{2n/3} g_j \cdot 3/2 \cdot u_j u_j^\top \in 3/2 \cdot K\Big] \\ & =& \gamma_F (K \cap F) \\ & \stackrel{\textrm{Lem~\ref{lem:ProbOfSubspaceVsBody}}}{\ge}& \gamma_H (K \cap H) \\ & \ge & e^{-\alpha n},
\end{eqnarray*}
where $(*)$ follows by Lemma~\ref{lem:ComparisonGaussians}. Then by Theorem~\ref{thm:partial_coloring}, the symmetric convex body $Q := \{x \in \setR^n : \sum_{j=1}^n x_j v_j \in K\}$ contains a good partial coloring in $Q \cap F$.
\end{proof}
Then Lemma~\ref{lem:partial_komlos_gen} implies the existence of a partial coloring
with optimal bounds as long as $n$ is of the order of $d$: 
\begin{corollary} \label{lem:PartialColoringForZonotope}
  Let $K \subseteq \setR^d$ be a normalized zonotope and let $v_1,\ldots,v_n \in K$. Then there is a randomized polynomial
  time algorithm to find a good partial coloring $x \in [-1,1]^n$ so that $\|\sum_{j=1}^n x_jv_j\|_K \lesssim \sqrt{d}$.
\end{corollary}

\begin{proof}
By Theorem~\ref{thm:SlicesOfNormZonotopes}, denoting $H := \spn\{v_1, \dots, v_n\}$, we have $\gamma_H (C \cdot K \cap H) \ge e^{-n}$. By Lemma~\ref{lem:ScalingLemma}, there exists some constant $\alpha > 0$ such that $\gamma_H(K \cap H) \ge e^{-\alpha n}$. By Lemma~\ref{lem:contained_ball}, $v_i \in \sqrt{d} B^d_2$, so that the statement follows directly from Lemma~\ref{lem:partial_komlos_gen}.
\end{proof}

\subsection{Proof of the main Theorem}

Now we have all the ingredients to prove our main result, Theorem~\ref{thm:VBofKwithK}.
\begin{proof}[Proof of Theorem~\ref{thm:VBofKwithK}]
  By Theorem~\ref{thm:sparsify}, we may assume that $K$ is generated by only $m \lesssim d \log d$ segments, and by Lemma~\ref{lem:normalized}, we may assume that $K$ is a normalized zonotope $K := \sqrt{\frac{d}{m}} A^\top B^m_\infty$ for some approximately regular $A \in \setR^{m \times d}$. By Theorem~\ref{thm:reduction_nd}, since $\vb(K,K) \le 2 \cdot \vb_d(K,K)$, we may assume that $n = d$, though for clarity we only use this in the final bound. As before we set  $Q := \{x \in \setR^n : \sum_{j=1}^n x_j v_j \in K\}$.
  We iteratively apply Lemma~\ref{lem:partial_komlos_gen} for $t$ rounds to obtain a partial coloring $x' \in Q \cap [-1,1]^n$, so that the set $I := \{i : |x'_i| < 1\}$ of partially colored indices satisfies $|I| \le n/2^t$, and by the triangle inequality over the $t$ rounds $\|\sum_{j=1}^n x'_j v_j\|_K \lesssim \sqrt{d} \cdot t$.

For each $j \in I$, we may write $v_j = \sqrt{\frac{d}{m}} A^\top u_i$ for some $u_i \in B_{\infty}^m$. By Theorem~\ref{thm:spencer}, we can find $\tilde{x} \in \setR^n$ so that $x := \tilde{x} + x' \in \{-1,1\}^n$ and $\sum_{i \in I} \tilde{x}_i u_i \in \sqrt{|I| \log(\frac{2m}{|I|})} \cdot c \cdot B^{m}_\infty$ where we set $\tilde{x}_i = 0$ for $i \notin I$. Therefore, setting $t := \log \log (\frac{2m}{n})$,
\begin{align*}
\Big\|\sum_{j=1}^n x_j v_j \Big\|_K & \le \Big\|\sum_{j=1}^n x'_j v_j \Big\|_K + \Big\|\sum_{j \in I} \tilde{x}_j v_j \Big\|_K \\ & \lesssim \sqrt{d} \cdot t + \sqrt{\frac{n}{2^t} \cdot \log\Big(\frac{2m}{n/2^t}\Big)} \\ & = \sqrt{d} \log \log \Big(\frac{2m}{n}\Big) + \underbrace{\sqrt{\frac{n}{\log(\frac{2m}{n})} \cdot \log\Big(\frac{2m}{n} \cdot \log \Big(\frac{2m}{n}\Big)\Big)}}_{\lesssim \sqrt{n} \le \sqrt{d}} \\ & \lesssim \sqrt{d} \log \log \Big(\frac{2m}{n}\Big) \\ & \lesssim \sqrt{d} \log \log \Big(\frac{d \log d}{n}\Big).\end{align*}

We conclude that $\vb(K,K) \lesssim \vb_d(K,K) \lesssim \sqrt{d} \log \log \log d$. \qedhere

\end{proof}

\section{The vector balancing constant $\mathrm{vb}(K,Q)$}
In this section we prove Theorem~\ref{thm:VBofKwithQ}, stating that $\vb(K,Q) \lesssim \sqrt{d \log d}$ where $K$ and $Q$ are normalized zonotopes.
First note that Cor~\ref{lem:PartialColoringForZonotope} indeed generalizes and for any $v_1,\ldots,v_n \in K$
there is a good partial coloring $x \in [-1,1]^n$ with $\|\sum_{j=1}^n x_jv_j\|_Q \lesssim \sqrt{d}$.
On the other hand, in the proof of Theorem~\ref{thm:VBofKwithK} we have also relied on Spencer's Theorem
which implies that $\vb_n(K,K) \lesssim \sqrt{n \log(\frac{2m}{n})}$. In particular this gives a bound that
improves as $n$ decreases. However in our setting with different zonotopes $K$ and $Q$ such a bound does not hold!

To see this, let $H \in \{ -1,1\}^{d \times d}$ be a Hadamard matrix, meaning that all rows and columns are
orthogonal. Then one can verify that $K := \frac{1}{\sqrt{d}}H^{\top}B_{\infty}^d$ is a normalized zonotope;
in fact, $K$ is a rotated cube. Fix any $n \leq d$ and consider the points $v_1,\ldots,v_n \in K$
with $v_i = \frac{1}{\sqrt{d}}H^{\top}H^i = \sqrt{d} \cdot e_i$. We choose $Q := B_{\infty}^d$ as the second normalized zonotope. Any good partial coloring $x \in [-1,1]^n$
must have a coordinate $i$ with $|x_i| \geq \frac{1}{2}$ and so $\|\sum_{j=1}^n x_jv_j\|_Q \geq \sqrt{d} |x_i| \geq \frac{\sqrt{d}}{2}$. 

Hence instead of applying Cor~\ref{lem:PartialColoringForZonotope} iteratively and obtaining a bound of $\vb(K,Q) \lesssim \sqrt{d} \log d$, we use Banaszczyk's Theorem together with Theorem~\ref{thm:SlicesOfNormZonotopes}:
\begin{proof}[Proof of Theorem~\ref{thm:VBofKwithQ}] Let $K,Q\subseteq\mathbb{R}^d$ be normalized zonotopes, and let $v_1,\ldots,v_n\in K$ be the vectors to be balanced. Define $H:=\mathrm{span}\{v_1, \ldots ,v_n\}$ and let  $r := \dim(H)\leq \min\{d,n\}$. By applying Theorem~\ref{thm:SlicesOfNormZonotopes} to the zonotope $Q$, subspace $H$, and $t:=\sqrt{2\log 2r}$, we find that
\begin{equation*}
    \gamma_H\big(\sqrt{2\log2r}C'Q\cap H\big)\geq e^{-\frac{1}{2}}>\frac{1}{2}.
\end{equation*}
By Lemma~\ref{lem:contained_ball} we know that $v_i\in \sqrt{d}B_2^d$ for each $i\in[n]$, hence by Theorem~\ref{thm:Banaszczyk}, signs $x\in\{-1,1\}^n$ can be computed in polynomial time such that
\begin{equation*}
  \sum_{j=1}^nx_jv_j\in \sqrt{d}C''\left(\sqrt{2\log2r}C'Q\cap H\right) 
  \subseteq C\sqrt{d\log \min\{d,n\}}Q,
\end{equation*}
as desired. In particular, $\mathrm{vb}(K,Q)\lesssim\sqrt{d\log d}$. \end{proof}

\section{Open problems\label{sec:OpenProblems}}

The main open question about zonotopes is whether a $d$-dimensional zonotope can be approximated up to a constant factor using only a linear number of segments: 

\begin{conjecture}[\cite{AIMWorkshop2007}] \label{conj: sparseZonotope}
For any zonotope $K \subseteq \setR^d$ and $0<\varepsilon \leq \frac{1}{2}$, does there exist a zonotope $Q$ with $O(\frac{d}{\varepsilon^2})$ segments so that $Q \subseteq K \subseteq (1+\varepsilon)Q$?
\end{conjecture}

Equivalently, since the polar body of a zonotope $A^\top B^m_\infty \subseteq \setR^d$ is the preimage $A^{-1} (B^m_1) := \{x \in \setR^d : \|Ax\|_1 \le 1\}$, we can restate the question as follows:

\begin{conjecture} \label{conj:SparseL1}
Does there exist a universal constant $C > 0$ such that given any matrix $A \in \setR^{m \times d}$ with $m \ge d$ and $0 < \varepsilon \leq \frac{1}{2}$, one can always find another matrix $\tilde{A} \in \setR^{Cd/\varepsilon^2 \times d}$ with $\|\tilde{A}x\|_1 \le \|Ax\|_1 \le (1+\varepsilon)\|\tilde{A}x\|_1$ for all $x \in \setR^d$?
\end{conjecture}

We remark that if one replaces the $\ell_1$ norm by the $\ell_2$ norm, an analogue of Conjecture~\ref{conj:SparseL1} holds as a direct corollary of a linear-size spectral sparsifier~\cite{TwiceRamanujanSparsifiers-BatsonSpielmanSrivastava-STOC09}. In that setting, each row of $\tilde{A}$ is a scalar multiple of a row of $A$, and there is hope that another rescaling of the rows of $A$ may suffice for the $\ell_1$ norm. Just as a spectral sparsifier can be found via spectral partial colorings~\cite{DBLP:conf/soda/ReisR20}, we also state the stronger conjecture of the existence of good partial colorings in the $\ell_1$ setting:

\begin{conjecture}
  Given any matrix $A \in \setR^{m \times d}$, does the set
  \[K := \Big\{x \in \setR^m : \Big|\sum_{i=1}^m x_i |\langle A_i, z\rangle| \Big| \le \sqrt{\tfrac{d}{m}} \|Az\|_1 \ \forall z \in \setR^d\Big\}\]
  have large Gaussian measure $\gamma_m (K) \ge e^{-Cm}$ where $C>0$ is a universal constant?
\end{conjecture}

Finally, we restate Schechtman's question, which would also follow from the above conjectures:

\begin{conjecture}[\cite{AIMWorkshop2007}]
Is it true that for any zonotope $K \subseteq \setR^d$, $\vb(K,K) \lesssim \sqrt{d}$?
\end{conjecture}
\bibliographystyle{alpha}
\bibliography{vectorBalancingConstantOfZonotopes}

\appendix

\section{Normalizing zonotopes\label{sec:NormalizingZonotopes}}

In this section, we show that for any full-dimensional zonotope $K \subseteq \setR^d$ there is a linear transformation $T : \setR^d \to \setR^d$ and a normalized zonotope $\tilde{K}$ so that $\frac{4}{5} \tilde{K} \subseteq T(K) \subseteq \tilde{K}$. 
For this result we will need the existence of \emph{Lewis weights} \cite{LewisWeights}:

\begin{theorem} Given a matrix $A \in \setR^{m \times d}$, there exists a unique vector $\overline{w} \in \setR^m_{> 0}$ so that for all $i \in [m]$ one has
\[ \overline{w}_i^{-2} A_i^\top (A^\top \overline{W}^{-1} A)^{-1} A_i = 1,\]
where $\overline{W} := \operatorname{diag}(\overline{w})$. Moreover, $\tr[\overline{W}] \le d$, with equality for full rank $A$.
\end{theorem}
Now to the proof of Lemma~\ref{lem:normalized}. 
\begin{proof}[Proof of Lemma~\ref{lem:normalized}]
  Consider a full-dimensional zonotope $K = A^{\top}B_{\infty}^m$ with $A \in \setR^{m \times d}$.
  Let $\overline{W}$ be the diagonal matrix corresponding to the Lewis weights of $A$ and let $W := D \overline{W}$ where $D > 0$
  is large enough so that $w_i := W_{i,i} \ge 1$ for all $i$. Define a matrix $B := A (A^\top W^{-1} A)^{-1/2} \in \setR^{m \times d}$ and define a second matrix $\tilde{A}$ where
  each row $B_i$ is replaced by $\lceil w_i \rceil$ many rows so that the first $\lfloor w_i \rfloor$ rows are all copies of $w_i^{-1} B_i$,
  and (if $\{w_i\}\neq 0$) the last row is $\{w_i\}^{1/2} w_i^{-1} B_i$, for a total of $m' := \sum_{i=1}^m \lceil w_i\rceil$ many rows. We will show that the conditions of Lemma~\ref{lem:normalized} hold with
  \begin{equation*}
      T:\mathbb{R}^d\rightarrow \mathbb{R}^d,\ \ T(K)=\sqrt{\tfrac{d}{m'}}(A^TW^{-1}A)^{-1/2}K=\sqrt{\tfrac{d}{m'}}\cdot B^TB_\infty^m
  \end{equation*}
  and $\widetilde{K}:=\sqrt{\tfrac{d}{m'}}\widetilde{A}^TB_\infty^{m'}$.
  
  First we show that $\tilde{K}$ is normalized, or equivalently that $\tilde{A}$ is approximately regular. Note that
\[
(\tilde{A}^\top \tilde{A})_{j,k} =  \sum_{i=1}^{m'} \tilde{A}_{i,j} \tilde{A}_{i,k} = \sum_{i=1}^m w_i^{-2} (\lfloor w_i \rfloor + \{w_i\}) \cdot B_{i,j} B_{i,k} = \sum_{i=1}^m w_i^{-1} \cdot B_{i,j} B_{i,k},
\]
so that by definition of $B$,
\[\tilde{A}^\top \tilde{A} = B^\top W^{-1} B = (A^\top W^{-1} A)^{-1/2} A^\top W^{-1} A (A^\top W^{-1} A)^{-1/2} = I_d.\]

Moreover, by the definition of Lewis weights, for each
row $i' \in [m']$ corresponding to a copy of $B_i$ one has
\[
\|\tilde{A}_{i'}\|_2^2\le w_i^{-2}B_i^\top B_i = w_i^{-2} A_i^\top (A^\top W^{-1} A)^{-1} A_i  = \frac{1}{D} \le \frac{2d}{m'},\]
where the last inequality follows since
\[m' = \sum_{i=1}^{m} \lceil w_i \rceil \le 2 \cdot \sum_{i=1}^m w_i = 2D \sum_{i=1}^m \overline{w}_i \le 2D \cdot d.\]
Thus $\widetilde{A}$ is approximately regular, and $\tilde{K}$ is normalized.

To see that $\frac{4}{5}\tilde{K}\subseteq T(K)\subseteq \tilde{K}$, take an arbitrary 
\begin{equation*}
    y = \tfrac{4}{5}\sqrt{\tfrac{d}{m'}}\sum_{i=1}^m \Big(\sum_{p = 1}^{\lfloor w_i\rfloor} x_{i,p} w^{-1}_i B_i + x_{i,\lceil w_i\rceil} \{w_i\}^{1/2} w^{-1}_i B_i\Big) \in \tfrac{4}{5}\sqrt{\tfrac{d}{m'}}\tilde{A}^\top B^{m'}_\infty=\tfrac{4}{5}\tilde{K},
\end{equation*}
and rewrite it as
\begin{equation*}
    \tfrac{4}{5}\sqrt{\tfrac{d}{m'}}\sum_{i=1}^m \Big(\underbrace{w_i^{-1} \Big(\sum_{i=1}^{\lfloor w_i\rfloor} x_{i,p} + x_{i, \lceil w_i \rceil} \{w_i\}\Big)}_{\in [-1,1]} + \underbrace{x_{i, \lceil w_i\rceil} \tfrac{\{w_i\}^{1/2} - \{w_i\}}{w_i}}_{\in [-\frac{1}{4},\frac{1}{4}]}\Big) B_i \in\sqrt{\tfrac{d}{m'}} B^\top B^m_\infty=T(K).
\end{equation*}
Now taking an arbitrary
$y := \sqrt{\tfrac{d}{m'}}\sum_{i=1}^m x_i B_i \in B^\top B^m_\infty=T(K)$, we may write 
\[y = \sqrt{\tfrac{d}{m'}}\sum_{i=1}^{m} \Big(\sum_{p = 1}^{\lfloor w_i\rfloor} x_i w^{-1}_i B_i + x_i \{w_i\} w^{-1}_i B_i\Big) \in \sqrt{\tfrac{d}{m'}}\tilde{A}^\top B^{m'}_\infty=\tilde{K}, \]
completing the proof of the lemma. Finally, note that this result immediately implies that \[
    \tfrac{4}{5}\vb(\tilde{K},\tilde{K})\leq \vb(K,K)\leq \tfrac{5}{4} \vb(\tilde{K},\tilde{K}). \qedhere
\]
\end{proof}

\section{Gaussian measure\label{sec:GaussianMeasure}}

\begin{proof}[Proof of Lemma~\ref{lem:GaussianMeasureOfStrip}]

We make use of the following tail inequality due to Szarek and Werner~\cite{SZAREK1999193} which holds for $t > -1$:
\[\Pr_{g \sim N(0,1)} [g > t] < \frac{1}{\sqrt{2 \pi}} \frac{4 e^{-t^2/2}}{3t + (t^2 + 8)^{1/2}}. \]
In particular, for $t \ge 1$ the right side is upper bounded by $\frac{1}{\sqrt{2 \pi}} \frac{4 e^{-t^2/2}}{6}$. Thus

\[\Pr_{g \sim N(0,1)} [|g| \le t] \ge 1 - \frac{4}{3 \sqrt{2 \pi}} e^{-t^2/2}. \]

Since the function $z \mapsto e^{-2z/3}$ is convex, we have $1 - \frac{4}{3\sqrt{2\pi}} z \ge e^{-2z/3}$ for all $z \in [0, e^{-1/2}]$ as it holds for the endpoints of the interval. Therefore for $t \ge 1$,

\[\Pr_{g \sim N(0,1)} [|g| \le t] \ge \exp(-\tfrac{2}{3} e^{-t^2/2}).\]

We conclude that for any $a \in \setR^n$ with $\|a\|_2 \leq 1$ and $t \geq 1$ one has
\[\Pr_{y \sim N(0,I_n)}[|\left<a,y\right>| \leq t] = \Pr_{g \sim N(0,1)} \Big[|g| \le \frac{t}{\|a\|_2}\Big] \ge \exp(-\tfrac{2}{3} e^{-t^2/(2\|a\|_2^2)}) \ge \exp(-e^{-t^2/2} \cdot \|a\|_2^2).\]

Indeed, the last inequality follows because

\[
  \frac{2}{3} \exp\Big( \frac{t^2}{2} - \frac{t^2}{2\|a\|_2^2}\Big) \le \frac{2}{3} \exp\Big(\frac{1}{2} - \frac{1}{2\|a\|_2^2}\Big) \le \frac{2}{3} \cdot e^{1/2} \cdot \frac{2}{e} \cdot \|a\|_2^2 \le \|a\|_2^2,\]
where the second to last inequality follows from $e^{z} \ge ez$ for $z := 1/(2\|a\|_2^2)$.
\end{proof}

\begin{proof}[Proof of Lemma~\ref{lem:ComparisonGaussians}]
Draw another random variable $z \sim N(0, B-A)$ and note that by log-concavity we have

\[\Pr_{y \sim N(0,A)} [y \in K] \ge \Pr_{y \sim N(0,A)} \Big[ \Pr_{z \sim N(0,B-A)} [y + z \in K]\Big] = \Pr_{y \sim N(0, B)} [y \in K]. \qedhere \]
\end{proof}
\end{document}